\documentclass[11pt]{article} 
\usepackage{amsfonts,amsmath,latexsym,amssymb,mathrsfs,amsthm,comment}
\usepackage{graphicx}
\usepackage{xcolor}
\usepackage{booktabs}

\usepackage[colorlinks=true, linkcolor=black, citecolor=black]{hyperref}

\numberwithin{equation}{section}

\let\OLDthebibliography\thebibliography
\renewcommand\thebibliography[1]{
  \OLDthebibliography{#1}
  \setlength{\parskip}{1pt}
  \setlength{\itemsep}{0pt plus 0.0ex}
}

\def\numberlikeadb{\global\def\theequation{\thesection.\arabic{equation}}}
\numberlikeadb
\newtheorem{theorem}{Theorem}[section]
\newtheorem{lemma}[theorem]{Lemma}
\newtheorem{corollary}[theorem]{Corollary}

\newtheorem{remark}[theorem]{Remark}

\usepackage{lscape}
\usepackage{caption}
\usepackage{multirow}
\allowdisplaybreaks
\begin{document}

\title{
Closed-form evaluation of the cumulative distribution function and absolute moments of the variance-gamma distribution
%On the cumulative distribution function and absolute moments of the variance-gamma distribution
%Closed-form formulas for the cumulative distribution function and absolute moments of the variance-gamma distribution
}
\author{Robert E. Gaunt\footnote{Department of Mathematics, The University of Manchester, Oxford Road, Manchester M13 9PL, UK, robert.gaunt@manchester.ac.uk; heather.sutcliffe@manchester.ac.uk}\:\, and Heather L. Sutcliff$\mathrm{e}^{*}$}

\date{} 
\maketitle

\vspace{-5mm}

\begin{abstract}Closed-form formulas are derived for the cumulative distribution function and odd-order absolute moments of the variance-gamma distribution. As a special case, closed-form expressions are deduced for the cumulative distribution function and odd-order absolute central moments of the product of two correlated zero-mean normal random variables. 
\end{abstract}



\section{Introduction}

The variance-gamma (VG) distribution with parameters $\nu > -1/2$, $0\leq|\beta|<\alpha$ and $\mu \in \mathbb{R}$ has probability density function (PDF)
\begin{equation}\label{vgpdf} f(x) = \frac{(\alpha^2-\beta^2)^{\nu+1/2}}{\sqrt{\pi}(2\alpha)^\nu \Gamma(\nu+1/2)} \mathrm{e}^{\beta (x-\mu)}|x-\mu|^{\nu}K_{\nu}\big(\alpha|x-\mu|\big), \quad x\in \mathbb{R},
\end{equation}
where $K_\nu(\cdot)$ denotes the modified Bessel function of the second kind (see \cite[Chapter 10]{olver}). If a random variable $X$ has PDF (\ref{vgpdf}), then we write $X\sim\mathrm{VG}(\nu,\alpha,\beta,\mu)$. The VG distribution is also referred to as the generalized Laplace distribution \cite{kkp01}, the Bessel function distribution \cite{m32} and the McKay Type II distribution \cite{ha04}. The VG distribution has been studied in the statistics literature since the work of \cite{p29} in 1929, and has been widely used in financial modelling following its introduction into the financial literature in the seminal papers \cite{madan2,madan}. Further applications are outlined in the recent review \cite{vg review}, along with a comprehensive account of the distributional theory of the VG distribution.

Whilst the distributional theory of the VG distribution is now well developed, a gap in the literature is that closed-form expressions for its cumulative distribution function (CDF) and absolute moments are available only for special parameter values (see \cite{g24} and \cite{g25}, respectively, and references therein). For general parameter values, the recent exact formulas obtained in \cite{g24} and \cite{g25} are expressed in terms of infinite series involving the modified Bessel function of the second kind and the modified Lommel function of the first kind rather than in closed form.

%Whilst the distributional theory of the VG distribution is now quite well developed, a clear gap in the literature is the fact that whilst closed-form formulas are available for the cumulative distribution function (CDF) and absolute moments for special parameter values (see \cite{g24} and \cite{g25}, respectively, and references therein), closed-form expressions are not available for these distributional properties for general parameter values. Exact formulas for the CDF and absolute moments of the VG distribution have recently been obtained in \cite{g24} and \cite{g25}, respectively; however, for general parameter values these expressions are not given in closed-form but rather in terms of an infinite series involving the modified Bessel function of the second kind and the modified Lommel function of the first kind. 

In this paper, we fill this gap by deriving closed-form formulas for the CDF and odd-order absolute moments of the VG distribution (see Theorems \ref{thm2.1} and \ref{thm2.2}), expressed in terms of the modified Bessel function of the second kind, the Gaussian hypergeometric function and the Kamp\'e de F\'eriet function. We restrict our attention to the problem of deriving a closed-form expression for the odd-order absolute moments because a closed-form expression for the even-order moments of the VG distribution is given in \cite[Proposition 2.1]{g25}. Our general formulas are obtained by evaluating the infinite series that appear in the formulas of  \cite{g24} and \cite{g25} for the CDF and odd-order absolute moments of the VG distribution (see Lemma \ref{lem2.3}).

An important special case of the VG distribution is the product of two correlated zero-mean normal random variables and, more generally, the sum of independent copies of such random variables \cite{g19}. Since the work of \cite{craig,wb32} in the 1930s, these distributions have received much interest in the statistics literature and have found numerous applications throughout the mathematical sciences; see \cite{g22,np16} for an overview of application areas and basic distributional theory. Since the product of two correlated zero-mean normal random variables is VG distributed, we obtain closed-form formulas for the CDF and odd-order absolute central moments of this distribution, which we present in Corollaries \ref{cor1} and \ref{cor2}.

\section{Results and proofs}
In the following two theorems, we provide closed-form formulas for the CDF and odd-order absolute moments of the VG distribution for general parameter values. Our formulas are expressed in terms of the modified Bessel function of the second kind, the  Gaussian hypergeometric function ${}_2F_1(a,b;c;x)$ (see \cite[Chapter 15]{olver}) 
%the Kamp\'e de F\'eriet function \cite{e78}, which is defined as
\begin{comment}
\begin{equation}
		F^{p:q\,;\,k}_{l:m\,;\,n}
		\!\left[
		\begin{array}{l}
			(a_1,\dots,a_p):(b_1,\dots,b_q);(c_1,\dots,c_k)
			\\[1mm]
			(\alpha_1,\dots,\alpha_l):(\beta_1,\dots,\beta_m);(\gamma_1,\dots,\gamma_n)
		\end{array}
		;\,
		x,y
		\right]
		=
		\sum_{r,s=0}^{\infty}
		\frac{
			\displaystyle
			\prod_{j=1}^{p}(a_j)_{r+s}
			\prod_{j=1}^{q}(b_j)_r
			\prod_{j=1}^{k}(c_j)_s
		}{
			\displaystyle
			\prod_{j=1}^{l}(\alpha_j)_{r+s}
			\prod_{j=1}^{m}(\beta_j)_r
			\prod_{j=1}^{n}(\gamma_j)_s
		}
		\frac{x^r}{r!}
		\frac{y^s}{s!}.\label{kdf}
	\end{equation}
\end{comment}
and the Kamp\'e de F\'eriet function \cite{k37}, which, in the notation of \cite[p.\ 423, equation (26)]{lond}, is defined by
\begin{align}\label{kdf} \mathrm{F}^{H:A;B}_{G:C;D}
	\Bigg[
	\begin{array}{c@{\ :\ }c@{\ ;\ }c}
		(h_H) & (a_A) & (b_B)\\
		(g_G) & (c_C) & (d_D)
	\end{array}
	\Bigg| x,y
	\Bigg]=\sum_{m=0}^\infty\sum_{n=0}^\infty\frac{((h_H))_{m+n}((a_A))_m((b_B))_n}{((g_G))_{m+n}((c_C))_m((d_D))_n}\frac{x^m}{m!}\frac{y^n}{n!}, 
   \end{align}
whenever the double series converges. Here, we write $(h_H)$ for the sequence $(h_1, h_2, \ldots , h_H )$, and $((h_H ))_n$ is defined by the product of Pochhammer symbols $((h_H))_n:=(h_1)_n(h_2)_n\cdots(h_H)_n$, $n\geq0$. We employ the convention that $((h_H))_0=1$. 
For further details on the Kamp\'e de F\'eriet function, we refer the reader to \cite[pp.\ 26--33]{sii}. We shall also let $\mathrm{sgn}(x)$ denote the sign function, $\mathrm{sgn}(x)=-1$ for $x<0$, $\mathrm{sgn}(0)=0$, $\mathrm{sgn}(x)=1$ for $x>0$. Finally, in order to state our results in a more compact form, we shall denote, for $|a|<1$, $b\geq1$, $\nu>-1/2$ and $y>0$,
\begin{align}
	S(a,b,\nu,y)&:=\frac{y^{\nu+b}}{2^{\nu+b-1}}\Bigg\{\frac{2}{b}\Bigg(K_\nu(y) 
	\mathrm{F}^{0:2;1}_{2:1;0}
	\Bigg[
	\begin{array}{c@{\;}c@{\;}c@{\;}c@{\;}c}
		- & : & \left(\frac{b}{2},\nu+\frac{b}{2}\right) & ; & 1 \\[2pt]
		\left(\frac{b+2}{2},\nu+\frac{b}{2}\right) & : & \frac{1}{2} & ; & -
	\end{array}
	\Bigg|
	\frac{a^2y^2}{4},\frac{y^2}{4}
	\Bigg]\nonumber\\
	&\quad+\frac{yK_{\nu-1}(y)}{2\nu+b}\mathrm{F}^{0:2;1}_{2:1;0} \Bigg[ \begin{array}{c@{\;}c@{\;}l@{\;}c@{\;}c}
		- &  : &  \left(\frac{b}{2},\nu+\frac{b}{2}\right) & ; & 1 \\[2pt]
		\left(\frac{b+2}{2},\nu+\frac{b+2}{2}\right) & : & \frac12 \, & ; &  -
	\end{array}\Bigg| \frac{a^2y^2}{4}, \frac{y^2}{4} \Bigg]\Bigg)\nonumber\\
	&\quad+\frac{2ay}{b+1}\Bigg(K_\nu(y)\mathrm{F}^{0:2;1}_{2:1;0}
	\Bigg[
	\begin{array}{c@{\;}c@{\;}c@{\;}c@{\;}c}
		- & : & \left(\frac{b+1}{2},\nu+\frac{b+1}{2}\right) & ; & 1 \\[2pt]
		\left(\frac{b+3}{2},\nu+\frac{b+1}{2}\right) & : & \frac{3}{2} & ; & -
	\end{array}
	\Bigg|
	\frac{a^2y^2}{4},\frac{y^2}{4}
	\Bigg]\nonumber\\
	&\quad+\frac{yK_{\nu-1}(y)}{2\nu+b+1}\mathrm{F}^{0:2;1}_{2:1;0}
	\Bigg[
	\begin{array}{c@{\;}c@{\;}c@{\;}c@{\;}c}
		- & : & \left(\frac{b+1}{2},\nu+\frac{b+1}{2}\right) & ; & 1 \\[2pt]
		\left(\frac{b+3}{2},\nu+\frac{b+3}{2}\right) & : & \frac{3}{2} & ; & -
	\end{array}
	\Bigg|
	\frac{a^2y^2}{4},\frac{y^2}{4}
	\Bigg]\Bigg)\Bigg\}.  \nonumber
\end{align}
For $y=0$, we define $S(a,b,\nu,0):=0$ by continuous extension. With this notation, we are ready to state our main results.
    
\begin{theorem}\label{thm2.1}Let $X\sim\mathrm{VG}(\nu,\alpha,\beta,\mu)$, where $\nu>-1/2$, $0\leq|\beta|<\alpha$, $\mu\in\mathbb{R}.$ Then, for 
$x\in\mathbb{R}$,
	\begin{align}
	%	F_X(x):=\mathbb{P}(X\leq x)&=\frac{1}{2}-\frac{\Gamma(\nu+1)}{\sqrt{\pi}\Gamma(\nu+1/2)}\frac{\beta}{\alpha}\bigg(1-\frac{\beta^2}{\alpha^2}\bigg)^{\nu+1/2}{}_2F_{1}\bigg(1,\nu+1;\frac32;\frac{\beta^2}{\alpha^2}\bigg)\nonumber\\
    %    &\quad+\mathrm{sgn}(x-\mu)\frac{(1-\beta^2/\alpha^2)^{\nu+1/2}}{2\sqrt{\pi}\Gamma(\nu+1/2)} S\bigg(\frac{2\beta}{\alpha}\mathrm{sgn}(x-\mu),1,\nu,\alpha|x-\mu|\bigg). \label{ncdf} 
  %      \\
        F_X(x):=\mathbb{P}(X\leq x)&=\frac{1}{2}-\frac{(1-\beta^2/\alpha^2)^{\nu+1/2}}{\sqrt{\pi}\Gamma(\nu+1/2)}\Bigg\{\frac{\beta}{\alpha}\Gamma(\nu+1)\,{}_2F_{1}\bigg(1,\nu+1;\frac32;\frac{\beta^2}{\alpha^2}\bigg)\nonumber\\
       &\quad-\frac{\mathrm{sgn}(x-\mu)}{2}S\bigg(\frac{\beta}{\alpha}\mathrm{sgn}(x-\mu),1,\nu,\alpha|x-\mu|\bigg)\Bigg\}. \nonumber
        %\label{ncdf}
	\end{align}
 \begin{comment}Fox-Wright function and the Kamp\'{e} de F\'{e}riet function as follows
	\begin{align}
		{}_p\Psi_q \!\left[
		\begin{matrix}
			(a_1,A_1),\ldots,(a_p,A_p)\\
			(b_1,B_1),\ldots,(b_q,B_q)
		\end{matrix}
		; z
		\right]
		=
		\sum_{k=0}^{\infty}
		\frac{
			\prod_{j=1}^{p}\Gamma(a_j+A_j k)
		}{
			\prod_{j=1}^{q}\Gamma(b_j+B_j k)
		}
		\frac{z^k}{k!},
	\end{align}
\end{comment}
\end{theorem}

\begin{theorem}\label{thm2.2}
	Let $X\sim\mathrm{VG}(\nu,\alpha,\beta,\mu)$, where $\nu>-1/2$, $0\leq|\beta|<\alpha$, $\mu\in\mathbb{R}$. Then, for every odd integer $r\geq1$,
	\begin{align}
		\mathbb{E}[|X|^r]&=\frac{2^r(1-\beta^2/\alpha^2)^{\nu+1/2}}{\sqrt{\pi}\alpha^r\Gamma(\nu+1/2)}\Bigg\{\sum_{k=0}^{(r-1)/2}\bigg(\frac{\alpha\mu}{2}\bigg)^{2k}\Gamma\bigg(\frac{r-2k+1}{2}\bigg)\Gamma\bigg(\nu+\frac{r-2k+1}{2}\bigg)\nonumber\\
		&\quad\times\bigg[\binom{r}{2k}{}_2F_{1}\bigg(\frac{r-2k+1}{2},\nu+\frac{r-2k+1}{2};\frac{1}{2};\frac{\beta^2}{\alpha^2}\bigg)\nonumber\\
		&\quad+\beta\mu\binom{r}{2k+1}{}_2F_1\bigg(\frac{r-2k+1}{2},\nu+\frac{r-2k+1}{2};\frac{3}{2};\frac{\beta^2}{\alpha^2}\bigg)\bigg]\nonumber\\
		&\quad-\sum_{k=0}^{r}\binom{r}{k}(-1)^k\bigg(\frac{\alpha|\mu|}{2}\bigg)^k S\bigg(-\frac{\beta}{\alpha}\mathrm{sgn}(\mu),r-k+1,\nu,\alpha|\mu|\bigg)\Bigg\}.\nonumber
	\end{align}
\end{theorem}

Before proving Theorems \ref{thm2.1} and \ref{thm2.2}, we recall the exact formulas of \cite{g24} and \cite{g25} for the CDF and odd-order absolute moments of the VG distribution.  The following exact formula for the CDF is given in \cite[Theorem 2.1]{g24}: for $x\in\mathbb{R}$,
	\begin{align}
		%F_X(x)&=\frac{1}{2}-\frac{\Gamma(\nu+1)}{\sqrt{\pi}\Gamma(\nu+1/2)}\frac{\beta}{\alpha}\bigg(1-\frac{\beta^2}{\alpha^2}\bigg)^{\nu+1/2}{}_2F_{1}\bigg(1,\nu+1;\frac32;\frac{\beta^2}{\alpha^2}\bigg)+\frac{(1-\beta^2/\alpha^2)^{\nu+1/2}}{2\sqrt{\pi}\Gamma(\nu+1/2)}\nonumber\\
	%	&\quad\times\sum_{k=0}^{\infty}\frac{(\mathrm{sgn}(x-\mu))^{k+1}}{k!}\bigg(\frac{2\beta}{\alpha}\bigg)^{k}\Gamma\bigg(\frac{k+1}{2}\bigg)\Gamma\bigg(\nu+\frac{k+1}{2}\bigg)G_{\nu+k,\nu}\big(\alpha|x-\mu|\big),\label{ogcdf} \\
   F_X(x)&=\frac{1}{2}-\frac{(1-\beta^2/\alpha^2)^{\nu+1/2}}{\sqrt{\pi}\Gamma(\nu+1/2)}\Bigg\{\frac{\beta}{\alpha}\Gamma(\nu+1)\,{}_2F_{1}\bigg(1,\nu+1;\frac32;\frac{\beta^2}{\alpha^2}\bigg)\nonumber\\
   &\quad-\frac{1}{2}\sum_{k=0}^{\infty}\frac{\big(\mathrm{sgn}(x-\mu)\big)^{k+1}}{k!}\bigg(\frac{2\beta}{\alpha}\bigg)^{k}\Gamma\bigg(\frac{k+1}{2}\bigg)\Gamma\bigg(\nu+\frac{k+1}{2}\bigg)G_{\nu+k,\nu}\big(\alpha|x-\mu|\big)\Bigg\}, \label{ogcdf}     
	\end{align}
where
		\begin{align}
		G_{\mu,\nu}(x)=x\big(K_\nu(x)\tilde{t}_{\mu-1,\nu-1}(x)+K_{\nu-1}(x)\tilde{t}_{\mu,\nu}(x)\big), \quad x>0, \label{g}
	\end{align}
with $G_{\mu,\nu}(0)=0$, and 
	\begin{align}
		\tilde{t}_{\mu,\nu}(x)=\frac{(x/2)^{\mu+1}}{\Gamma\big(\frac{\mu-\nu+3}{2}\big)\Gamma\big(\frac{\mu+\nu+3}{2}\big)}{}_1F_2\bigg(1;\frac{\mu-\nu+3}{2};\frac{\mu+\nu+3}{2};\frac{x^2}{4}\bigg)\label{t}
	\end{align}
is a normalisation of the modified Lommel function of the first kind that was introduced in \cite{g20}. Here ${}_1F_2(a;b,c;x)$ is a generalized hypergeometric function (see \cite[Chapter 16]{olver}). The following exact formula for odd-order absolute moments of the VG distribution is given in \cite[Theorem 2.3]{g25}: for every odd integer $r\geq1$,
	\begin{align}
		\mathbb{E}[|X|^r]&=\frac{2^r(1-\beta^2/\alpha^2)^{\nu+1/2}}{\sqrt{\pi}\alpha^r\Gamma(\nu+1/2)}\Bigg\{\sum_{k=0}^{(r-1)/2}\bigg(\frac{\alpha\mu}{2}\bigg)^{2k}\Gamma\bigg(\frac{r-2k+1}{2}\bigg)\Gamma\bigg(\nu+\frac{r-2k+1}{2}\bigg)\nonumber\\
		&\quad\times\bigg[\binom{r}{2k}{}_2F_{1}\bigg(\frac{r-2k+1}{2},\nu+\frac{r-2k+1}{2};\frac{1}{2};\frac{\beta^2}{\alpha^2}\bigg)\nonumber\\
		&\quad+\beta\mu\binom{r}{2k+1}{}_2F_1\bigg(\frac{r-2k+1}{2},\nu+\frac{r-2k+1}{2};\frac{3}{2};\frac{\beta^2}{\alpha^2}\bigg)\bigg]\nonumber\\
		&\quad-\sum_{k=0}^{r}\binom{r}{k}(-1)^k\bigg(\frac{\alpha|\mu|}{2}\bigg)^k\sum_{j=0}^{\infty}\frac{\big(-2\mathrm{sgn}(\mu)\beta/\alpha\big)^j}{j!}\Gamma\bigg(\frac{r+j-k+1}{2}\bigg)\nonumber\\
		&\quad\times\Gamma\bigg(\nu+\frac{r+j-k+1}{2}\bigg)G_{\nu+r+j-k,\nu}\big(\alpha|\mu|\big)\Bigg\}.\label{ogam}
	\end{align}

We will prove Theorems \ref{thm2.1} and \ref{thm2.2} by applying the following lemma to the expressions (\ref{ogcdf}) and (\ref{ogam}).     

\begin{lemma}\label{lem2.3} For $|a|<1$, $b\geq1$, $\nu>-1/2$ and $y\geq0$,
\begin{align}
\mathcal{S}(a,b,\nu,y):=\sum_{j=0}^\infty\frac{(2a)^j}{j!}\Gamma\bigg(\frac{j+b}{2}\bigg)\Gamma\bigg(\nu+\frac{j+b}{2}\bigg)G_{\nu+j+b-1,\nu}(y)=S(a,b,\nu,y). \label{lemfor} 
\end{align}     
\end{lemma}

\begin{proof} For $y=0$, formula (\ref{lemfor}) follows immediately since $G_{\nu+j+b-1,\nu}(0)=S(a,b,\nu,0)=0$. We therefore suppose from now on that $y>0$. We note that, for $y>0$, the series (\ref{lemfor}) is absolutely convergent (see \cite[Remark 2.7]{g25}). %which shows that the series is absolutely convergent in the case that $b$ is a positive integer, with the extension to $b\geq1$ being clear.
We begin by applying the formulas (\ref{g}) and (\ref{t}) to write $\mathcal{S}(a,b,\nu,y)$ in terms of the generalized hypergeometric function:
	\begin{align}
		\mathcal{S}(a,b,\nu,y)	&=\frac{y^{\nu+b}}{2^{\nu+b-1}}\sum_{j=0}^{\infty}\frac{(ay)^j}{j!}\Gamma\bigg(\frac{j+b}{2}\bigg)\Gamma\bigg(\nu+\frac{j+b}{2}\bigg)\frac{1}{\Gamma\big(\frac{j+b+2}{2}\big)}\nonumber\\
		&\quad\times\Bigg(\frac{K_\nu(y)}{\Gamma\big(\nu+\frac{j+b}{2}\big)}{}_1F_2\bigg(1;\frac{j+b+2}{2},\nu+\frac{j+b}{2};\frac{y^2}{4}\bigg)\nonumber\\
		&\quad+\frac{yK_{\nu-1}(y)}{2\Gamma\big(\nu+\frac{j+b+2}{2}\big)}{}_1F_2\bigg(1;\frac{j+b+2}{2},\nu+\frac{j+b+2}{2};\frac{y^2}{4}\bigg)\Bigg). \nonumber
	\end{align}
Using the power series representation of the generalized hypergeometric function
%\begin{equation}
	%	\nonumber
	%	{}_pF_q(a_1,\ldots,a_p; b_1,\ldots,b_q;x)=\sum_{j=0}^\infty\frac{(a_1)_j\cdots(a_p)_j}{(b_1)_j\cdots(b_q)_j}\frac{x^j}{j!},
	%\end{equation}
	(see \cite[Eq.\ 16.2.1]{olver}), we obtain
	\begin{align}
		\mathcal{S}(a,b,\nu,y)	&=\frac{y^{\nu+b}}{2^{\nu+b-1}}\sum_{	j=0}^{\infty}\sum_{l=0}^{\infty}\frac{(ay)^j}{j!}\frac{(y^2/4)^l}{l!}\Gamma\bigg(\frac{j+b}{2}\bigg)\Gamma\bigg(\nu+\frac{j+b}{2}\bigg)\frac{(1)_l}{\Gamma\big(\frac{j+b+2}{2}\big)\big(\frac{j+b+2}{2}\big)_l}\nonumber\\
		&\quad\times
		\Bigg(\frac{K_\nu(y)}{\Gamma\big(\nu+\frac{j+b}{2}\big)\big(\nu+\frac{j+b}{2}\big)_l}+\frac{yK_{\nu-1}(y)}{2\Gamma\big(\nu+\frac{j+b+2}{2}\big)\big(\nu+\frac{j+b+2}{2}\big)_l}\Bigg). \nonumber
	\end{align}
	Splitting the $j$-sum into its even and odd parts by setting $j=2m$ and $j=2m+1$ gives that
	\begin{align}
		\mathcal{S}(a,b,\nu,y)	&=\frac{y^{\nu+b}}{2^{\nu+b-1}}\sum_{m=0}^{\infty}\sum_{l=0}^{\infty}{(a^2y^2)^m}\frac{(y^2/4)^l}{l!}\Bigg\{\Gamma\bigg(m+\frac{b}{2}\bigg)\Gamma\bigg(\nu+m+\frac{b}{2}\bigg)\nonumber\\
		&\quad\times\frac{(1)_l}{(2m)!\Gamma\big(m+1+\frac{b}{2}\big)\big(m+1+\frac{b}{2}\big)_l}\Bigg(\frac{K_\nu(y)}{\Gamma\big(\nu+m+\frac{b}{2}\big)\big(\nu+m+\frac{b}{2}\big)_l}\nonumber\\
		&\quad+\frac{yK_{\nu-1}(y)}{2\Gamma\big(\nu+m+1+\frac{b}{2}\big)\big(\nu+m+1+\frac{b}{2}\big)_l}\Bigg)+ay\Gamma\bigg(m+\frac{b+1}{2}\bigg)\Gamma\bigg(\nu+m+\frac{b+1}{2}\bigg)\nonumber\\
		&\quad\times\frac{(1)_l}{(2m+1)!\Gamma\big(m+\frac{b+3}{2}\big)\big(m+\frac{b+3}{2}\big)_l}\Bigg(\frac{K_\nu(y)}{\Gamma\big(\nu+m+\frac{b+1}{2}\big)\big(\nu+m+\frac{b+1}{2}\big)_l}\nonumber\\
		&\quad+\frac{yK_{\nu-1}(y)}{2\Gamma\big(\nu+m+\frac{b+3}{2}\big)\big(\nu+m+\frac{b+3}{2}\big)_l}\Bigg)
		\Bigg\}	.\label{mo1}
	\end{align}
	We can now express the above in terms of four Kamp\'{e} de F\'{e}riet functions. Below we include the argument for the first term; the other three terms are dealt with similarly. We have 
	\begin{align}
		&\sum_{m=0}^{\infty}\sum_{l=0}^{\infty}\frac{(a^2y^2)^m}{(2m)!}\frac{(y^2/4)^l}{l!}\Gamma\bigg(m+\frac{b}{2}\bigg)
		\frac{(1)_lK_\nu(y)}{\Gamma\big(m+1+\frac{b}{2}\big)\big(m+1+\frac{b}{2}\big)_l\big(\nu+m+\frac{b}{2}\big)_l}\nonumber\\
		&=K_\nu(y)\sum_{m=0}^{\infty}\sum_{l=0}^{\infty}\frac{\big(a^2y^2/4\big)^m}{m!(1/2)_m}\frac{(y^2/4)^l}{l!}
		\frac{(1)_l}{\big(m+\frac{b}{2}\big)\big(m+1+\frac{b}{2}\big)_l\big(\nu+m+\frac{b}{2}\big)_l}\nonumber\\
		&=\frac{2K_\nu(y)}{b}\sum_{m=0}^{\infty}\sum_{l=0}^{\infty}\frac{(a^2y^2/4)^m(y^2/4)^l}{m!(1/2)_ml!}\frac{(1)_l\big(\frac{b}{2}\big)_m}{\big(\frac{b}{2}+1\big)_m\big(m+1+\frac{b}{2}\big)_l\big(\nu+m+\frac{b}{2}\big)_l}\nonumber\\
		&=\frac{2K_\nu(y)}{b}\sum_{m=0}^{\infty}\sum_{l=0}^{\infty}\frac{(a^2y^2/4)^m(y^2/4)^l}{m!(1/2)_ml!}\frac{(1)_l\big(\frac{b}{2}\big)_m\big(\nu+\frac{b}{2}\big)_m}{\big(\frac{b}{2}+1\big)_{m+l}\big(\nu+\frac{b}{2}\big)_{m+l}}\nonumber\\
		&=\frac{2K_\nu(y)}{b}\mathrm{F}^{0:2;1}_{2:1;0}
		\Bigg[
		\begin{array}{c@{\;}c@{\;}c@{\;}c@{\;}c}
			- & : & \left(\frac{b}{2},\nu+\frac{b}{2}\right) & ; & 1 \\[2pt]
			\left(\frac{b}{2}+1,\nu+\frac{b}{2}\right) & : & \frac{1}{2} & ; & -
		\end{array}
		\Bigg|
		\frac{a^2y^2}{4},\frac{y^2}{4}
		\Bigg],\,\nonumber
	\end{align}
	where in the first step we used the standard identity $\Gamma(z+1)=z\Gamma(z)$ and the relation $(2m)!=2^{2m}\,m!(1/2)_m$, in the second step we used that $m+b/2=(b/2)(b/2+1)_m/(b/2)_m$, in the third step we used the relation $(c+d)_n=(c)_{d+n}/(c)_d$, and in the final step we used the definition (\ref{kdf}) of the Kamp\'{e} de F\'{e}riet function. Applying the same argument to the other three terms in (\ref{mo1}) yields the desired formula (\ref{lemfor}), completing the proof. 
    %gives the desired expression (\ref{lemfor}) for \(S(a,b,\nu,y)\), completing the proof. 
\end{proof}

\begin{proof}[Proof of Theorem \ref{thm2.1}] Apply Lemma \ref{lem2.3} to formula (\ref{ogcdf}) with $a=\beta \,\mathrm{sgn}(x-\mu)/\alpha$, $b=1$ and $y=\alpha|x-\mu|$.
\end{proof}

\begin{proof}[Proof of Theorem \ref{thm2.2}] Apply Lemma \ref{lem2.3} to formula (\ref{ogam}) with $a=-\beta\,\mathrm{sgn}(\mu)/\alpha$, $b=r-k+1$ and $y=\alpha|\mu|$.
\end{proof}

Consider the product $Z=UV$, where $(U,V)$ is a zero-mean bivariate normal random vector with positive variances $(\sigma_U^2,\sigma_V^2)$ and correlation coefficient $\rho\in(-1,1)$. Also, for a positive integer $n$, let $Z_1,\ldots,Z_n$ be independent copies of $Z$. It was shown in \cite{g19} that the product $Z$ and, more generally, the sum $S_n=\sum_{i=1}^{n}Z_i$ are VG distributed. In particular,
\begin{align}
	S_n\sim \mathrm{VG}\bigg(\frac{n-1}{2},\frac{1}{s(1-\rho^2)},\frac{\rho}{s(1-\rho^2)},0\bigg),\label{vg1}
\end{align}
where $s=\sigma_U\sigma_V$. From the distributional relation (\ref{vg1}), the following corollaries are immediate consequences of Theorems \ref{thm2.1} and \ref{thm2.2}. In Corollary \ref{cor2}, we provide closed-form formulas for the odd-order absolute central moments of $S_n$; we remark that a closed-form formula for the absolute moments of $S_n$ (with no centring) is given in \cite[Corollary 2.4]{g23}.  

\begin{corollary}\label{cor1} For $x\in\mathbb{R}$,
\begin{align}
F_{S_n}(x)&=\frac{1}{2}-\frac{\Gamma\big((n+1)/2\big)}{\sqrt{\pi}\Gamma(n/2)}\rho(1-\rho^2)^{n/2}\,{}_2F_1\bigg(1,\frac{n+1}{2};\frac{3}{2};\rho^2\bigg)\nonumber\\
&\quad+\frac{(1-\rho^2)^{n/2}}{2\sqrt{\pi}\Gamma(n/2)}\mathrm{sgn}(x) \,S\bigg(\rho\,\mathrm{sgn}(x),1,\frac{n-1}{2},\frac{|x|}{s(1-\rho^2)}\bigg).\label{en1}
\end{align}
\end{corollary}

\begin{proof}
This is immediate from Theorem \ref{thm2.1} and the distributional relation (\ref{vg1}). 
\end{proof}

\begin{corollary}\label{cor2} For every odd integer $r\geq1$,
\begin{align}\mathbb{E}[|S_n-\mathbb{E}[S_n]|^r]&=\frac{(2s)^r(1-\rho^2)^{n/2+r}}{\sqrt{\pi}\Gamma(n/2)}\Bigg\{\sum_{k=0}^{(r-1)/2}\bigg(\frac{-n\rho}{2(1-\rho^2)}\bigg)^{2k}\Gamma\bigg(\frac{r-2k+1}{2}\bigg)\Gamma\bigg(\frac{r-2k+n}{2}\bigg)\nonumber\\
		&\quad\times\bigg[\binom{r}{2k}{}_2F_{1}\bigg(\frac{r-2k+1}{2},\frac{r-2k+n}{2};\frac{1}{2};\rho^2\bigg)\nonumber\\
		&\quad-\frac{n\rho^2}{(1-\rho^2)}\binom{r}{2k+1}{}_2F_1\bigg(\frac{r-2k+1}{2},\frac{r-2k+n}{2};\frac{3}{2};\rho^2\bigg)\bigg]\nonumber\\
		&\quad-\sum_{k=0}^{r}\binom{r}{k}\bigg(\frac{-n|\rho|}{2(1-\rho^2)}\bigg)^k S\bigg(|\rho|,r-k+1,\frac{n-1}{2},\frac{n|\rho|}{1-\rho^2}\bigg)\Bigg\}.\label{en2}
        \end{align}
\end{corollary}

\begin{proof} From the basic fact that $\mathbb{E}[S_n]=n\rho s$, we have that
\[S_n-\mathbb{E}[S_n]\sim \mathrm{VG}\bigg(\frac{n-1}{2},\frac{1}{s(1-\rho^2)},\frac{\rho}{s(1-\rho^2)},-n\rho s\bigg),\]
and combining this distributional relation with Theorem \ref{thm2.2} yields formula (\ref{en2}).
\end{proof}

\begin{remark}
Setting $n=1$ in formulas (\ref{en1}) and (\ref{en2}) yields expressions for the CDF and odd-order absolute central moments of the product $Z$. For the CDF, we use that $\Gamma(1)=1$, $\Gamma(1/2)=\sqrt{\pi}$ and
\begin{equation}{}_2F_1\bigg(1,1;\frac{3}{2};z\bigg)=\frac{\sin^{-1}(\sqrt{z})}{\sqrt{z(1-z)}}, \quad 0\leq z<1,\label{ffffu}
\end{equation}
(see \cite[Eqs.\ 15.4.4 and 15.8.1]{olver}), where the right-hand side of (\ref{ffffu}) is interpreted by continuity at $z=0$, with value 1. This gives the simplification
\begin{align*}
F_Z(x)&=\frac{1}{2}-\frac{1}{\pi}\sin^{-1}(\rho)+\frac{\sqrt{1-\rho^2}}{2\pi}\mathrm{sgn}(x) \,S\bigg(\rho\,\mathrm{sgn}(x),1,0,\frac{|x|}{s(1-\rho^2)}\bigg), \quad x\in\mathbb{R}.   
\end{align*}
\end{remark}

\section*{Acknowledgements} Both authors are funded by EPSRC grant EP/Y008650/1.


\footnotesize

\begin{thebibliography}{99}
\addcontentsline{toc}{section}{References}

%\bibitem{barn1} Barndorff-Nielsen, O. E. and Stelzer, R. Absolute Moments of Generalized Hyperbolic Distributions and Approximate Scaling of Normal Inverse Gaussian L\'{e}vy Processes. \emph{Scand. J. Stat.} $\mathbf{32}$ (2005),  617--637.

\bibitem{craig} Craig, C. C. On the Frequency Function of $xy$. \emph{Ann. Math. Stat.} $\mathbf{7}$ (1936), 1--15.

%\bibitem{e78} Exton, H. {\it Handbook of Hypergeometric Integrals: Theory, Applications, Tables, Computer Programs}. Halsted Press, New York, 1978. 

\bibitem{vg review} Fischer, A., Gaunt, R. E. and Sarantsev, A. The Variance-Gamma Distribution: A Review. \emph{Stat. Sci.} $\mathbf{40}$ (2025), 235--258.

%\bibitem{g14} Gaunt, R.E. Variance-gamma approximation via Stein's method, \emph{Electron. J. Probab.} 19 (38) (2014).

\bibitem{g19} Gaunt, R. E. A note on the distribution of the product of zero mean correlated normal random variables. \emph{Stat. Neerl.} $\mathbf{73}$ (2019),  176--179.

\bibitem{g20} Gaunt, R. E.  Bounds for modified Lommel functions of the first kind and their ratios. \emph{J. Math. Anal. Appl.} $\mathbf{486}$ (2020), Art.\ 123893.

\bibitem{g22} Gaunt, R. E. The basic distributional theory for the product of zero mean correlated normal random variables. \emph{Stat. Neerl.} $\mathbf{76}$ (2022), 450--470.

\bibitem{g23} Gaunt, R. E. On the moments of the variance-gamma distribution. \emph{Stat. Probabil. Lett.} $\mathbf{201}$ (2023), Art.\ 109884.

\bibitem{g24} Gaunt, R. E. On the cumulative distribution function of the variance-gamma distribution. \emph{Bull. Aust. Math. Soc.} $\mathbf{110}$ (2024), 389--397.

\bibitem{g25} Gaunt, R. E. Absolute moments of the variance-gamma distribution. \emph{J. Math. Anal. Appl.} $\mathbf{543}$ (2025), Art.\ 128861.

\bibitem{ha04} Holm, H. and Alouini, M.--S. Sum and difference of two squared correlated Nakagami variates in connection with the McKay distribution. \emph{IEEE Trans. Commun.} $\mathbf{52}$ (2004),  1367--1376.

%\bibitem{jp21} Jankov Ma\v{s}irevi\'c, D. and Pog\'{a}ny, T. K. On new formulae for cumulative distribution function for McKay Bessel distribution. \emph{Commun. Stat. Theory} $\mathbf{50}$ (2021), 143--160.

\bibitem{k37} Kamp\'e de F\'eriet, J. La fonction hyperg\'eom\'etrique, M\'emorial des sciences math\'ematiques, vol. 85, Paris: Gauthier-Villars, 1937.

\bibitem{kkp01} Kotz, S., Kozubowski, T. J. and Podg\'{o}rski, K. \emph{The Laplace Distribution and Generalizations: A Revisit with Applications to Communications, Economics, Engineering, and Finance.} Springer, 2001.

\bibitem{madan2} Madan, D. B., Carr, P. P. and Chang, E. C. The variance gamma process and option pricing. \emph{Eur. Finance Rev.} $\mathbf{2}$ (1998), 79--105.

\bibitem{madan} Madan, D. B. and Seneta, E. The Variance Gamma (V.G.) Model for Share Market Returns. \emph{J. Bus.} $\mathbf{63}$ (1990), 511--524.

\bibitem{m32} McKay, A. T. A Bessel function distribution. \emph{Biometrika} $\mathbf{24}$ (1932),  39--44.

\bibitem{np16} Nadarajah, S. and Pog\'any, T. K. On the distribution of the product of correlated normal random variables. \emph{C. R. Math.} $\mathbf{354}$ (2016), 201--204.

%\bibitem{nsg07} Nadarajah, S., Srivastava H. M. and Gupta, A. K. Skewed Bessel function distributions with application to rainfall data. \emph{Statistics} $\mathbf{41}$ (2007), 333--344.

%\bibitem{o21} Ogasawara, H. A non-recursive formula for various moments of the multivariate normal distribution with sectional truncation. \emph{J. Multivariate Anal.} $\mathbf{183}$ (2021), Art.\ 104792.

\bibitem{olver} Olver, F. W. J., Lozier, D. W., Boisvert, R. F. and Clark, C. W.  \emph{NIST Handbook of Mathematical Functions.} Cambridge University Press, 2010.

\bibitem{p29} Pearson, K., Jeffery, G. B. and Elderton, E. M. On the distribution of the first product moment-coefficient, in samples drawn from an indefinitely large normal population. \emph{Biometrika} $\mathbf{21}$ (1929), 164--193.

%\bibitem{scott} Scott, D. J., W\"urtz, D., Dong, C. and Tran, T. T.  Moments of the generalized hyperbolic distribution. \emph{Computation. Stat.} $\mathbf{26}$ (2011),  459--476.

\bibitem{sii} Srivastava, H. M. and Karlsson, P. W. \emph{Multiple Gaussian Hypergeometric Series}; Ellis Horwood Series: Mathematics and Its Applications;
Ellis Horwood Ltd.: Chichester, UK; Halsted Press [John Wiley \& Sons, Inc.]: New York, USA, 1985.

\bibitem{lond} Srivastava, H. M. and Panda, R. An integral representation for the product of two Jacobi polynomials. \emph{J. Lond. Math. Soc.} $\mathbf{12}$ (1976), 419--425.

%\bibitem{w12} Winkelbauer, A. Moments and Absolute Moments of the Normal Distribution. arXiv:1209.4340, 2012.

\bibitem{wb32} Wishart, J. and Bartlett, M. S. The distribution of second order moment statistics in a normal system. \emph{Proc. Camb. Philol. Soc.} $\mathbf{28}$ (1932), 455--459.

\end{thebibliography}
\end{document}